\documentclass[a4paper, 12pt]{article}
\usepackage{fullpage}
\usepackage{amsmath, amsthm, amssymb, mathrsfs, graphicx, subfigure}
\usepackage{ifthen, url}
\usepackage[usenames]{color}

\usepackage[sort&compress]{natbib}
\bibpunct{(}{)}{;}{a}{}{,}

\theoremstyle{plain}
\newtheorem{thm}{\indent Theorem}
\newtheorem{cor}{\indent Corollary}
\newtheorem{lem}{\indent Lemma}

\theoremstyle{remark}

\theoremstyle{definition}
\newtheorem{defn}{\indent Definition}

\theoremstyle{definition}
\newtheorem{ex}{\indent \it Example}

\theoremstyle{definition}
\newtheorem{assumption}{\indent \it Assumption}

\theoremstyle{definition}
\newtheorem*{PRalgorithm}{\indent PR algorithm}

\theoremstyle{plain}
\newtheorem*{chenthm}{\indent Chen's Theorem}

\newcommand{\A}{\mathscr{A}}

\newcommand{\U}{\mathscr{U}}
\newcommand{\Ubar}{\overline{\mathscr{U}}}
\newcommand{\Y}{\mathscr{Y}}

\newcommand{\FF}{\mathbb{F}}
\newcommand{\MM}{\mathbb{M}}
\newcommand{\RR}{\mathbb{R}}

\newcommand{\E}{\mathsf{E}}

\newcommand{\grad}{\nabla}

\newcommand{\eps}{\varepsilon}
\renewcommand{\phi}{\varphi}

\begin{document}

\title{Convergence rate for predictive recursion estimation of finite mixtures}
\author{
Ryan Martin \\ 
Department of Mathematics, Statistics, and Computer Science \\ 
University of Illinois at Chicago \\ 
\url{rgmartin@math.uic.edu} 
}
\date{\today}

\maketitle

\begin{abstract}
Predictive recursion (PR) is a fast stochastic algorithm for nonparametric estimation of mixing distributions in mixture models.  It is known that the PR estimates of both the mixing and mixture densities are consistent under fairly mild conditions, but currently very little is known about the rate of convergence.  Here I first investigate asymptotic convergence properties of the PR estimate under model misspecification in the special case of finite mixtures with known support.  Tools from stochastic approximation theory are used to prove that the PR estimates converge, to the best Kullback--Leibler approximation, at a nearly root-$n$ rate.  When the support is unknown, PR can be used to construct an objective function which, when optimized, yields an estimate the support.  I apply the known-support results to derive a rate of convergence for this modified PR estimate in the unknown support case, which compares favorably to known optimal rates.  

\medskip

\end{abstract}

\section{Introduction}
\label{S:intro}

Nonparametric estimation of mixing distributions is an important and challenging problem in statistics.  Recent progress along these lines has been made with the fast stochastic \emph{predictive recursion} (PR) algorithm due to \citet{nqz} and \citet{newton02}.  PR is fundamentally different from existing algorithms, such as EM, in a number of ways.  Most importantly, PR is not a hill-climbing algorithm.  Instead, it learns sequentially like stochastic approximation \citep{robbinsmonro, kushner}.  In addition, PR is able to estimate a mixing density with respect to any user-defined dominating measure.  That is, unlike the nonparmetric maximum likelihood estimate, which is almost surely discrete \citep{lindsay1995}, the PR estimate can be discrete, continuous, or both, depending on the user's choice of dominating measure.  

Theoretically, it has been shown that the PR estimates of both the mixing and mixture densities are consistent under certain conditions; see Section~\ref{S:pr} for more details.  The goal of this note is to investigate the rate of convergence, about which very little is known.  For this, we shall explore further the connection between PR and stochastic approximation developed in \citet{martinghosh}.  To the author's knowledge, results on the rate of convergence for general stochastic approximations are only fully developed in the finite-dimensional context.  Therefore, we shall confine ourselves here to an analysis of PR when the possibly misspecified model assumes that the data-generating distribution is a finite mixture with known support.  In this case, we prove that the PR estimate of the mixing distribution converges almost surely at a nearly parametric root-$n$ rate, where the limit is characterized by the mixture model closest to the true data-generating distribution based on the Kullback--Leibler divergence.  This result also sheds light on how one should choose PR's tuning parameter in practical applications.  

The PR algorithm itself is not naturally suited for the case when the support of the finite mixture model is unknown.  But, by applying the general principle in \citet{mt-prml}, I show that PR yields a sort of objective function which can be optimized to estimate the unknown support.  I apply the paper's known-support results to establish rates of convergence for this new PR-based unknown-support procedure.  Two numerical examples are given to illustrate the method; for more examples and the full computational details, the reader is referred to \citet{prml-finite}.

\section{Predictive recursion}
\label{S:pr}

Suppose independent data $Y_1,\ldots,Y_n$ are available from a distribution with unknown density $m(y)$, which we model as a nonparametric mixture:
\begin{equation}
\label{eq:mixture}
m_f(y) = \int_{\U} p(y \mid u) f(u) \,d\mu(u), \quad y \in \Y, 
\end{equation}
where $(y,u) \mapsto p(y \mid u)$ is a known kernel on $\Y \times \U$ and $f \in \FF$ is unknown and to be estimated.  Here $\FF = \FF(\U,\mu)$ is the set of all densities with respect to a given $\sigma$-finite Borel measure $\mu$ on $\U$.  \citet{newton02} presents the following algorithm for nonparametric estimation of $f$ and $m_f$ based on $Y_1,\ldots,Y_n$.  

\begin{PRalgorithm}
Choose a density $f_0 \in \FF$ and a sequence of weights $\{w_i: i \geq 1\} \subset (0,1)$.  Then, for $i=1,\ldots,n$, compute $m_{i-1}(y) = m_{f_{i-1}}(y)$ and 
\begin{equation}
\label{eq:recursion}
f_i(u) = (1-w_i) f_{i-1}(u) + w_i p(Y_i \mid u) f_{i-1}(u) \, / \, m_{i-1}(Y_i).
\end{equation}
Return $f_n(u)$ and $m_n(y)=m_{f_n}(y)$ as estimates of $f(u)$ and $m_f(y)$, respectively.  
\end{PRalgorithm}

PR has some interesting connections to the nonparametric Bayes estimate in the case where the unknown mixing distribution is modeled as a random draw from the Dirichlet process distribution.  \citet{mt-prml} take advantage of this connection to motivate a PR-based semiparametric mixture model analysis where an additional unknown structural parameter is estimated by maximizing a PR-induced approximate marginal likelihood.  \citet{mt-test} use this general strategy to develop a PR-based methodology for large-scale nonparametric empirical Bayes multiple testing.  In Section~\ref{S:prml} I apply this method to mixtures with unknown support.  

Asymptotic convergence properties of the PR estimates $f_n$ and $m_n$ have only recently become available.  Let $\MM$ denote the set of mixture densities $m_f$ as $f$ ranges over $\FF$.  \citet{tmg} build on the work of \cite{ghoshtokdar} to show that when the mixture model is correctly specified (i.e., $m \in \MM$), then both $f_n$ and $m_n$ converge almost surely to $f$ and $m_f$ in their respective topologies.  \citet{mt-rate} go one step further, showing that if $m \not\in \MM$, then $m_n$ converges to the closest mixture density $m_{f^\star} \in \MM$ as measured by the Kullback--Leibler divergence.  As a corollary, if $f$ is identifiable in the postulated mixture model, then $f_n$ converges almost surely to $f^\star$ in the weak topology.  They also establish a bound on the rate of convergence for $m_n$ in terms of the PR weight sequence $\{w_n\}$.  For weights of the form $w_i = (i+1)^{-\gamma}$, for suitable $\gamma$, \citet{mt-rate} obtain a $n^{-1/6}$ bound on the Hellinger convergence rate of $m_n$ to $m_{f^\star}$ for a wide class of kernels $p(y \mid u)$.  While this rate is comparable to the rate obtained in \citet{genovese2000}, it leaves a lot to be desired.  In fact, simulations in \citet{mt-rate} suggest that the upper bound corresponds to a ``worst case scenario'' rate of convergence, i.e., when $f^\star$ sits on the boundary of $\FF$.  I expect that a nearly parametric root-$n$ rate for $m_n$, like that obtained by \citet{ghosalvaart2001}, can be achieved by PR, at least in some cases.  In Section~\ref{S:rate} we show that this conjecture holds in the special known finite support case.

\section{Asymptotics for PR with known support}
\label{S:rate}

Assume that the true density $m$ is modeled as a finite mixture.  That is, $\U$ is a finite set of size $s$ and $\mu$ is counting measure.  In this case, $\FF$ denotes the $(s-1)$-dimensional probability simplex, and I write $f = \{f(u): u \in \U\}$.  Then $m_f(y) = \sum_{u \in \U} p(y \mid u) f(u)$.  Throughout, all $s$-dimensional vectors $x$ will be indexed by $\U$, i.e., $x = \{x(u): u \in \U\}$.  Also, $\langle \cdot, \cdot \rangle$ denotes the usual inner-product and $\|\cdot\|$ the corresponding norm.  

We begin by listing two basic assumptions about the mixture model.  

\begin{assumption}
\label{as:kernel}
$u \mapsto p(y \mid u)$ is continuous for each $y \in \Y$.    
\end{assumption}

\begin{assumption}
\label{as:identifiable}
$f$ is identifiable in model \eqref{eq:mixture}, i.e., $f \mapsto m_f$ is one-to-one.  
\end{assumption}

For any density $m'$ on $\Y$, define the Kullback--Leibler divergence of $m'$ from $m$ as $K(m,m') = \int \log\{m(y)/m'(y)\} m(y) \,dy$.  Henceforth, I shall silently assume that $K(m,m') < \infty$ for all $m' \in \MM$.  Then the infimum 
\[ K^\star = \inf\{K(m,m_f): f \in \FF\}, \]
is finite.  It follows from Assumption~\ref{as:kernel} that there exists an $f^\star$ in the closure of $\FF$ such that $K(m,m_{f^\star}) = K^\star$; see Lemma~3.1 of \citet{mt-rate}.  Assumption~\ref{as:identifiable} ensures that $f^\star$ is unique.  Allowing the model to be misspecified is particularly important here, given that the assumption of known finite support is rather strong.  For example, even if the support $\U$ is unknown, the results that follow show that PR does as well asymptotically as could be hoped for if we simply guess at what $\U$ should be.  

Following \citet{martinghosh}, express the PR update $f_{n-1} \mapsto f_n$, $n \geq 1$, as follows:
\begin{equation}
\label{eq:pr1}
f_n(u) = f_{n-1}(u) + w_n \Phi(Y_n, f_{n-1})(u),  \quad u \in \U, 
\end{equation}
where, for generic $y \in \Y$ and $f \in \FF$, the mapping $\Phi(y, f)$ is defined as 
\[ \Phi(y, f)(u) = f(u) \Bigl\{ \frac{p(y \mid u)}{m_f(y)} - 1 \Bigr\}. \]
Equation \eqref{eq:pr1} shows that PR is a special case of a general Robbins--Monro type of stochastic approximation algorithm designed to find roots of the mapping 
\begin{equation}
\label{eq:phimap}
\phi(f)(u) = f(u) \Bigl\{ \int \frac{p(y \mid u)}{m_f(y)} m(y) \,dy - 1 \Bigr\}, \quad f \in \FF, \quad u \in \U. 
\end{equation}
This $\phi(f)$ is nothing but the conditional expectation of $\Phi(Y_n, f_{n-1})$, under the true density $m$, given $f_{n-1}$ equals $f$.  The following result is an immediate consequence of the definitions and construction above.

\begin{lem}
\label{lem:martingale}
The sequence $Z_n(u)$, for $u \in \U$, given by 
\begin{equation}
\label{eq:martingale}
Z_n(u) = \Phi(Y_n, f_{n-1})(u) - \phi(f_{n-1})(u), 
\end{equation}
is a martingale difference sequence with respect to the $\sigma$-algebra $\A_n$ generated by $Y_1,\ldots,Y_n$.  Moreover, $\|Z_n\|^2$ is bounded for all $n \geq 1$.  
\end{lem}

According to stochastic approximation theory \citep[e.g.,][]{kushner}, convergence properties of $f_n$, as $n \to \infty$, can be found by investigating the asymptotic behavior of solutions of an appropriate ordinary differential equation (ODE).  Specifically, let $\{f^t: t \geq 0\}$ denote a generic trajectory in $\FF$.  Then the limiting behavior of solutions $f^t$ of the ODE $df^t/dt = \phi(f^t)$, as $t \to \infty$, can be used to study the limiting behavior of the PR sequence $f_n$, as $n \to \infty$.  For this purpose, I will need some basic definitions and results from the theory of ODEs.  

\begin{lem}
\label{lem:equilibrium}
The mixing distribution $f^\star$ is an equilibrium point of the ODE $df^t/dt = \phi(f^t)$; in other words, $\phi(f^\star)(u) = 0$ for all $u$.  
\end{lem}

\begin{proof}
Plugging $f^\star$ into the expression in \eqref{eq:phimap} gives 
\[ \phi(f^\star)(u) = f^\star(u) \Bigl\{ \int \frac{p(y \mid u)}{m_{f^\star}(y)} m(y) \,dy - 1 \Bigr\}. \]
By the fact that $f^\star$ minimizes $K(m,m_f)$, it follows from Lemma~3.3 of \citet{mt-rate} that $\phi(f^\star)(u) \leq 0$ for each $u$.  But since $\sum_{u} \phi(f^\star)(u)$ vanishes, it must be that $\phi(f^\star)(u) = 0$ for each $u$, proving the claim.  
\end{proof}

The goal is to show that $f^\star$ is a stable equilibrium in the sense that any solution to the ODE converges to $f^\star$, regardless of the initial condition.  For this, a \emph{Lyapunov function} will be useful.  

\begin{defn}
\label{def:lyapunov}
A function $\ell: \FF \to \RR$ is a Lyapunov function at $f^\star$ for the ODE $df^t/dt = \phi(f^t)$ if (i) $\ell(f)$ is continuously differentiable in a neighborhood of $f^\star$, (ii) $\ell(f) \geq 0$ with equality if and only if $f = f^\star$, and (iii) $\dot\ell(f) = \langle \grad \ell(f), \phi(f) \rangle \leq 0$.  
\end{defn}

Lyapunov's theory, described beautifully in \citet{lasalle}, states that if a Lyapunov function $\ell(f)$ exists at $f = f^\star$, then $f^\star$ is a stable equilibrium point.  Next I show that a slight variation of the Kullback--Leibler divergence is a Lyapunov function in the present context.  

\begin{lem}
\label{lem:lyapunov}
The mapping $\ell: \FF \to [0,\infty)$ given by 
\begin{equation}
\label{eq:kl}
\textstyle \ell(f) = K(m,m_f) - K^\star + \sum_u f(u) - 1
\end{equation}
is a Lyapunov function for the ODE $df^t/dt = \phi(f^t)$.  
\end{lem}

\begin{proof}
Properties (i) and (ii) in Definition~\ref{def:lyapunov} are obvious.  For property (iii), simple calculus reveals that $\phi(f)(u) = -f(u) \{ \grad \ell(f) \}(u)$, from which it follows that $\dot\ell(f) = -\sum_u f(u) \{ \grad \ell(f) \}(u)^2 \leq 0$.  That equality is obtained if and only if $f=f^\star$ follows from the fact that $f^\star$ is the unique minimizer of $K(m,m_f)$ and, hence, the only point at which $\grad \ell(f)$ vanishes.  
\end{proof}

The function $\ell(f)$ in \eqref{eq:kl} can be viewed as a Lagrange multiplier version of the Kullback--Leibler divergence with the trivial constraint $\sum_u f(u) = 1$.  This is consistent with the interpretation of PR as an algorithm that asymptotically minimizes $K(m,m_f)$ over $\FF$ \citep{mt-rate}.  Another important observation, used in Lemma~\ref{lem:derivative} below, is that $\ell(f)$ is convex.  

Next I state an extension of the PR convergence theorem in \citet{martinghosh} for the case where the true data-generating density $m$ need not belong to the class $\MM$ of mixture models \eqref{eq:mixture}.  For this we need

\begin{assumption}
\label{as:weights}
$\sum_n w_n = \infty$ and $\sum_n w_n^{1+\eps} < \infty$ for some $\eps \in (0,1]$.  
\end{assumption}

In practice, it is common to take $w_n = (n+1)^{-\gamma}$ for $\gamma \in (1/2, 1]$.  Then Assumption~\ref{as:weights} holds with $\eps > \gamma^{-1} - 1$.  

\begin{thm}
\label{thm:mg}
Under Assumptions~\ref{as:kernel}--\ref{as:weights}, $f_n \to f^\star$ almost surely, where $f^\star$ is the unique minimizer of $K(m,m_f)$ over $\FF$.    
\end{thm}

\begin{proof}
In light of Lemmas~\ref{lem:martingale}--\ref{lem:lyapunov}, the claim follows from Theorem~5.2.3 of \citet{kushner} and the continuity of $\phi(f)$; see \citet{martinghosh}.  
\end{proof}

The main result on a rate of convergence for PR will make use of a general theorem on convergence rates of stochastic approximation \citep[][Theorem~3.1.1]{chen2002}; see Appendix~\ref{S:chen}.  But two preliminary result are needed first. 

\begin{lem}
\label{lem:martingale2}
The sequence $Z_n$ in \eqref{eq:martingale} satisfies $\sum_{n=1}^\infty w_n^{1-\delta} Z_n < \infty$ almost surely for $\delta \in ( 0, (1-\eps)/2 ]$, where $\eps$ is as in Assumption~\ref{as:weights}.  
\end{lem}

\begin{proof}
Let $X_N = \sum_{n=1}^N w_n^{1-\delta} Z_n$.  By Lemma~\ref{lem:martingale}, $\{X_N: N \geq 1\}$ is a martingale sequence and, since $\{Z_n\}$ is bounded, 
\[ \E \|X_N\|^2 = \sum_{n=1}^N w_n^{2(1-\delta)} \E\|Z_n\|^2 \leq \text{const} \cdot \sum_{n=1}^\infty w_n^{2(1-\delta)}. \]
Taking $\delta \leq (1-\eps)/2$, it follows from Assumption~\ref{as:weights} that $\E\|X_N\|^2$ is uniformly bounded in $N$.  Then the martingale convergence theorem \citep[][Theorem~5.14]{breiman1992} implies that $X_N$ converges almost surely, completing the proof.  
\end{proof}

An additional assumption about the weights is required.  For weights given by $w_n = (n+1)^{-\gamma}$, this assumption holds as long as $\gamma < 1$.  

\begin{assumption}
\label{as:weights2}
$\{w_n\}$ satisfies $w_{n+1}^{-1} - w_n^{-1} \to 0$.  
\end{assumption}

\begin{lem}
\label{lem:derivative}
Let $J = D\phi(f^\star)$ denote the derivative of $\phi$ evaluated at $f=f^\star$.  If $f^\star$ is in the interior of $\FF$, then all eigenvalues of $J$ are negative.
\end{lem}

\begin{proof}
Simple calculus reveals that $J = D\phi(f^\star)$ is of the form 
\[ J(u,v) = -f^\star(u) \int \frac{p(y \mid u)p(y \mid v)}{m_{f^\star}(y)^2} m(y) \,dy, \quad u,v \in \U. \]
In matrix notation, write $J = -\text{diag}(f^\star) \cdot \grad^2 \ell(f^\star)$, where $\text{diag}(f^\star)$ is a diagonal matrix with the elements of $f^\star$ as its diagonal entries, and $\grad^2 \ell(f^\star)$ is the second derivative matrix of $\ell(f)$ evaluated at $f = f^\star$.  Since $f^\star$ is in the interior of $\FF$, all entries are positive and, hence, $\text{diag}(f^\star)$ is positive definite.  Since $\ell(f)$ is convex on $\FF$, $\grad^2 \ell(f^\star)$ is also positive definite.  The claim follows from the fact that the product of these two positive definite matrices, which is $-J$, must have positive eigenvalues.    
\end{proof}

An interesting observation is that the matrix $P = -J^{\top}$, the negative transpose of the Jacobian $J$ in Lemma~\ref{lem:derivative}, is a transition probability matrix for an irreducible, aperiodic Markov chain on $\U$.  This chain is also reversible and has $f^\star$ as its stationary distribution.  But how this observation might be useful in studying the asymptotic convergence of PR remains unclear.  

In light of Assumptions~\ref{as:kernel}--\ref{as:weights2}, Lemmas~\ref{lem:martingale2} and \ref{lem:derivative}, and the existence of a Lyapunov function proved in Lemma~\ref{lem:lyapunov}, the main result on the convergence rate of PR is a consequence of Chen's theorem in Appendix~\ref{S:chen}.

\begin{thm}
\label{thm:rate}
Assume that $f^\star$ lies in the interior of $\FF$.  Then under Assumptions~\ref{as:kernel}--\ref{as:weights2}, $\|f_n - f^\star\| = o(w_n^\delta)$ almost surely for $\delta$ in Lemma~\ref{lem:martingale2}.  
\end{thm}

When the weights are given by $w_n = (n+1)^{-\gamma}$, for $\gamma \in (1/2,1)$, it follows from Theorem~\ref{thm:rate} and the previous discussion that $\|f_n-f^\star\| = o(n^{-(1-1/2\gamma)})$ almost surely.  Since $\gamma$ can be chosen arbitrarily close to 1, it follows that the convergence rate can be made arbitrarily close to $n^{-1/2}$ almost surely.  

A slightly stronger version of Theorem~\ref{thm:rate} could be obtained if weight sequences were allowed to satisfy $w_{n+1}^{-1} - w_n^{-1} \to \alpha$, with $\alpha > 0$.  For example, if $w_n = (n+1)^{-1}$, then $\alpha=1$.  This extension would make the root-$n$ rate possible, but it would require all eigenvalues of $J$ in Lemma~\ref{lem:derivative} to be less than $-1/2$.  At this point it is unclear whether this claim is true; standard bounds for eigenvalues, such as those in Gershgorin's theorem or Proposition~2 in \citet{diaconis.stroock.1991}, are not helpful in this case.  

Almost sure rates of convergence for the mixture density $m_n$ to $m_{f^\star}$ are available as consequences of Theorem~\ref{thm:rate}.  The $L_1$ rate follows immediately from its definition, while the rate for the Kullback--Leibler contrast, $K(m,m_n) - K^\star$, requires a simple second-order Taylor approximation of $\ell(f)$ at $f = f^\star$.  

\begin{cor}
\label{cor:mixture-rate}
Under the conditions on Theorem~\ref{thm:rate}, $\int |m_n-m_{f^\star}|\,dy = o(w_n^\delta)$ almost surely for $\delta$ in Lemma~\ref{lem:martingale2}.  Likewise, $K(m,m_n)-K^\star = o(w_n^{2\delta})$.  
\end{cor}

\citet{mt-rate} derive a bound of $o(W_n^{-1})$ for $K(m,m_n)-K^\star$ in the general compact $\U$ case, where $W_n = \sum_{i=1}^n w_i$.  When $w_n = (n+1)^{-\gamma}$, the bound for $K(m,m_n)-K^\star$ in \citet{mt-rate} becomes $o(n^{-(1-\gamma)})$, which can be no faster than $n^{-1/3}$ under their conditions.  Compare this to the rate of $o(n^{-(2-1/\gamma)})$ obtained from Corollary~\ref{cor:mixture-rate}, which is considerably faster than $n^{-1/3}$ for $\gamma \approx 1$, albeit for the special known finite support case.  So, regarding the PR weights $\{w_i: i \geq 1\}$, the message here, contrary to that in \citet{mt-rate}, is that the faster the weights vanish the faster the overall convergence.

\section{PR with unknown support}
\label{S:prml}

The PR convergence theory in the previous section assumes the finite support is known and only the mixing distribution is unknown.  In practice, however, both the support and mixing distribution are unknown and to be estimated.  To close this gap, I propose here a new PR-based approach for handling the unknown support case.  The asymptotic results in Section~\ref{S:rate} will be used to prove consistency of this new procedure.  Two simple examples are also given for illustration, but the computational details, simulations, and extensions will be presented elsewhere \citep{prml-finite}.  

\subsection{Setup}

Let $\Ubar$ be a compact set, large enough that there is a finite mixture supported in $\Ubar$ that gives a sufficiently accurate approximation to $m$.  Take $U$ to be a generic finite subset of $\Ubar$.  By treating $U$ as the fixed support, a run of PR will produce a sequence of estimates $\{(f_{i,U}, m_{i,U}): i=1,\ldots,n\}$ of the mixing and mixture distributions, whose dependence on the chosen support set $U$ are now made explicit.  In the same vein, write $\FF_U$ for the $(|U|-1)$-dimensional probability simplex and define $K^\star(U) = \inf\{K(m,m_{f,U}): f \in \FF_U\}$, the smallest Kullback--Leibler number for mixtures supported on $U$.  

The jumping off point is that the result $K(m,m_{n,U}) - K^\star(U) = o(w_n^{2\delta})$ of Corollary~\ref{cor:mixture-rate} holds ``pointwise'' for all $U$; that is, the particular support $U$ plays no role in the analysis of Section~\ref{S:rate}.  Thus, in the present case where the support is unknown, a reasonable strategy is to estimate the support by minimizing, over $U$, some estimate of $K(m,m_{n,U})$.  This is the approach advocated by \citet{mt-prml}.  Indeed, by making connections to PR and Dirichlet process mixture models, they argue that, in the present context, the appropriate estimate of $K(m,m_{n,U})$ is 
\begin{equation}
\label{eq:ku}
K_n(U) = \sum_{i=1}^n \log \frac{m(Y_i)}{m_{i-1,U}(Y_i)}, \quad U \subset \Ubar, \quad |U| < \infty. 
\end{equation}
Then the goal is to minimize $K_n(U)$ over $U$.  But since it is not possible to perform this optimization over all finite $U \subset \Ubar$, some adjustment must be made.  Consider starting with a fixed finite subset $\U$ of $\Ubar$ obtained by chopping up $\Ubar$ into a sufficiently fine grid, so that $|\U|$ is large.  Then the collection of all subsets $U$ of $\U$ is huge---it has $2^{|\U|}-1$ elements---but finite so it is possible to minimize $K_n(U)$ over $U \subseteq \U$.  \citet{prml-finite} uses a simulated annealing strategy to perform this optimization.  Once the minimizer $\hat U_n$ of $K_n(U)$ is obtained, PR is run once more to produce $f_{n,\hat U_n}$ and $m_{n,\hat U_n}$ as estimates of the mixing and mixture distributions, respectively.

\subsection{Large-sample theory}

For simplicity, I will assume that the true density $m$ is indeed a mixture density of the postulated form with support contained in $\U$; the more general case can be handled similarly, but with an additional technical assumption \citep[][Assumption~6]{mt-prml}.  Also, assume that $w_n = (n+1)^{-\gamma}$ for some $\gamma \in (0.5,1)$.  To get convergence of the approximation $K_n(U)$ to $K^\star(U)$, I will need one additional assumption, stated next, which holds for many common kernels, including normal and Poisson.

\begin{assumption}
\label{as:lr.bound}
There exists a finite constant $A > 0$ such that 
\[ \max_{u_1,u_2,u_3 \in \U} \int \Bigl\{ \frac{p(y \mid u_1)}{p(y \mid u_2)} \Bigr\}^2 p(y \mid u_3) \,dy \leq A. \]
\end{assumption}

Under Assumptions~\ref{as:kernel}--\ref{as:lr.bound}, one can follow the proof of Theorem~2 in \citet{mt-prml} to conclude that, for each fixed $U \subseteq \U$, 
\begin{equation}
\label{eq:limit}
\lim_{n \to \infty} \Bigl| c_n\bigl\{K_n(U) - K^\star(U)\bigr\} - \frac{c_n}{n}\sum_{i=1}^n \bigl\{K(m,m_{i-1,U}) - K^\star(U) \bigr\} \Bigr| = 0,
\end{equation}
almost surely, for any sequence $c_n$ that satisfies $c_n = O(n^{1/2-\eps})$ for some $\eps > 0$.  It follows from Corollary~\ref{cor:mixture-rate} that the summation in \eqref{eq:limit} is of the order $n^{1/\gamma-1}$.  So, if $\eps > \max\{0,\gamma^{-1}-3/2\}$, the right-most term in the modulus in \eqref{eq:limit} vanishes and, therefore, so must the left-most term.  This proves that, for $\gamma \approx 1$, $K_n(U) \to K^\star(U)$ pointwise in $U$ at a rate just slower than $n^{-1/2}$.  But since $2^{\U}$ is finite, the convergence is also uniform.  The following theorem summarizes this result.  

\begin{thm}
\label{thm:mprmle}
Choose weights $w_n = (n+1)^{-\gamma}$ with $\gamma \in (0.5,1)$ and let $\eps > \max\{0,\gamma^{-1}-3/2\}$.  Then, under Assumptions~\ref{as:kernel}--\ref{as:lr.bound}, $n^{1/2-\eps}\{K_n(U)-K^\star(U)\} \to 0$ almost surely as $n \to \infty$.  Moreover, since $U$ ranges only over a finite set, $n^{1/2-\eps}K_n(\hat U_n) \to 0 = K^\star(U^\star)$, where $U^\star \subseteq \U$ is the support of the true mixture distribution.   
\end{thm}

If I define a distance $d$ between two sets as the cardinality of their symmetric difference, then Theorem~\ref{thm:mprmle} states that $d(\hat U_n, U^\star) = o(n^{-1/2+\eps})$.  In other words, $\hat U_n$ is a nearly root-$n$ $d$-consistent estimate of $U^\star$.  Furthermore, a nearly root-$n$ rate of convergence for $f_{n,\hat U_n}$ can be obtained, which I now sketch.  With a slight abuse of notation, I can bound the total variation distance between $f_{n,\hat U_n}$ and $f^\star$ as follows:
\begin{align*}
d_{\text{\sc tv}}(f_{n,\hat U_n},f^\star) & = \sum_{u \in \U} |f_{n,\hat U_n}(u) - f^\star(u)| \\
& = \sum_{u \in \hat U_n \cap U^{\star c}} f_{n,\hat U_n}(u) + \sum_{u \in \hat U_n^c \cap U^\star} f^\star(u) + \sum_{u \in \hat U_n \cap U^\star} |f_{n,\hat U_n}(u) - f^\star(u)| \\
& \leq d(\hat U_n, U^\star) + d_{\text{\sc tv}}(f_{n,\hat U_n}, f_{n,U^\star}) + d_{\text{\sc tv}}(f_{n,U^\star},f^\star). 
\end{align*}
The two outer-most terms on the right-hand side vanish at a nearly root-$n$ rate according to Theorems~\ref{thm:mprmle} and \ref{thm:rate}, respectively.  The middle term is more difficult to analyze, but it is clear that the data-dependent PR mapping $U \mapsto f_{n,U}$ is, in some sense, continuous in $U$.  So, the convergence of $d_{\text{\sc tv}}(f_{n,\hat U_n},f_{n,U^\star})$ is also driven by $d(\hat U_n, U^\star)$.  Therefore, the rate for $d_{\text{\sc tv}}(f_{n,\hat U_n},f^\star)$ must also be nearly $n^{-1/2}$.  

Recall that \citet{chen1995} showed that, for finite mixtures, the optimal rate of convergence is $n^{-1/4}$.  In that case, the unknown finite support is allowed to be anything, essentially nonparametric, so the rates are relatively slow.  In contrast, by restricting the set of candidate supports to subsets of a large but ultimately finite set $\U$, I am able to achieve a nearly parametric root-$n$ rate of convergence.

\subsection{Examples}

Here I give two relatively simple real-data examples---a Gaussian location mixture and a Poisson mixture---to illustrate the potential of the proposed method.   

\begin{ex}
\label{ex:galaxy1}
Under the Big Bang model, galaxies should form clusters and the relative velocities of the galaxies should be similar within clusters.  \citet{roeder} considers velocity data for $n = 82$ galaxies.  She models this data as a finite Gaussian mixture, with the number and location of the mixture components unknown.  The assumption is that each galactic cluster is a single component of the Gaussian mixture.  The presence of multiple mixture components is consistent with the hypothesis of galaxy clustering.  

We apply the methodology outlined above to estimate the mixing distribution $f$.  We will consider a simple Gaussian mixture model in which each component has variance $\sigma^2 = 1$, based on the \emph{a priori} considerations of \citet{escobar.west.1995}.  From the observed velocities, it is apparent that the mixture components should be centered somewhere in the interval $\Ubar = [5,40]$, so we choose a grid of candidate support points $\U = \{5.0,5.5,6.0,\ldots,39.5,40.0\}$.  Figure~\ref{fig:galaxy} shows the corresponding estimates of the mixing and mixture distribution.   The PR method identifies six galaxy clusters, and the estimates of $U$ and $f$ closely match those of \citet{ishwaran.james.sun.2001} and others.  
\end{ex}

\begin{figure}
\begin{center}
\subfigure[Mixing distribution]{\scalebox{0.6}{\includegraphics{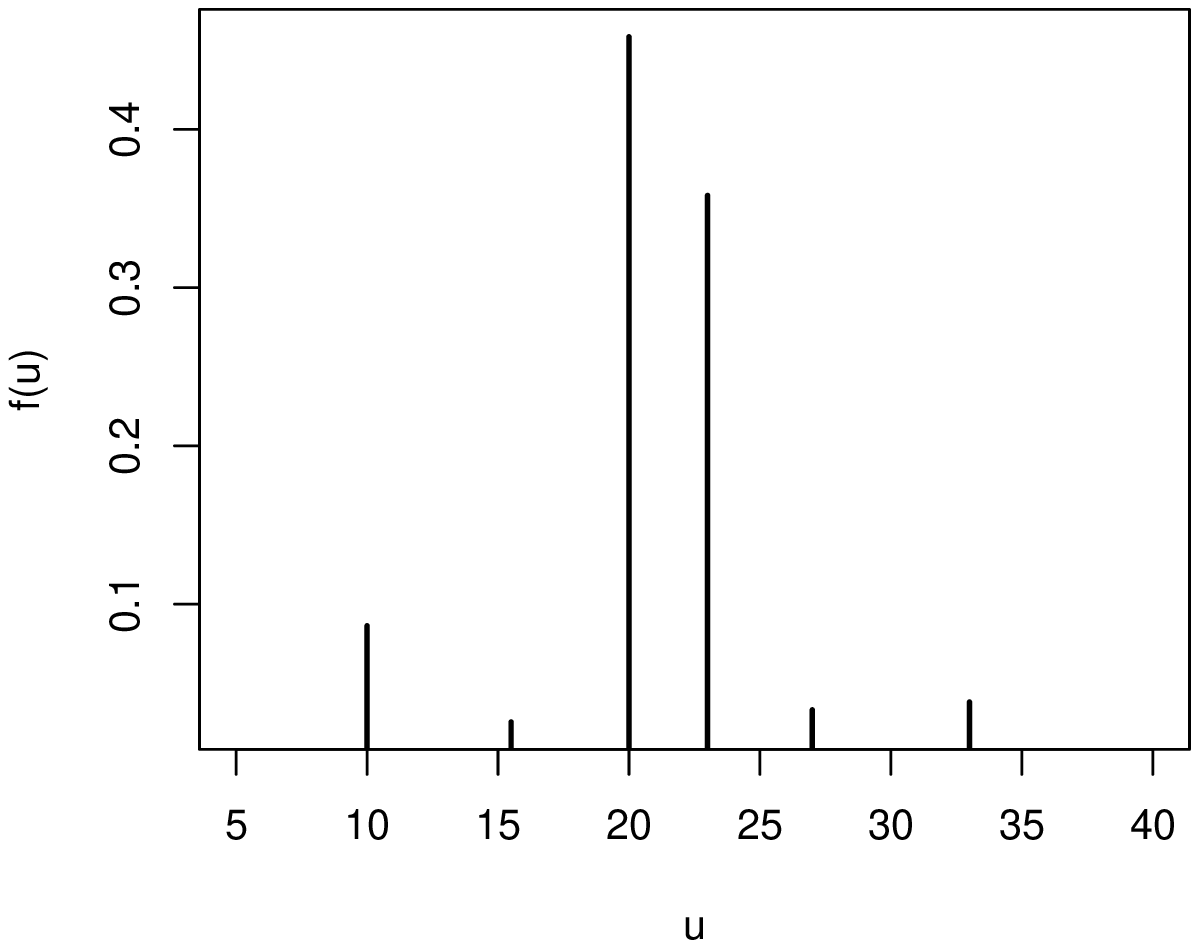}}}
\subfigure[Mixture distribution]{\scalebox{0.6}{\includegraphics{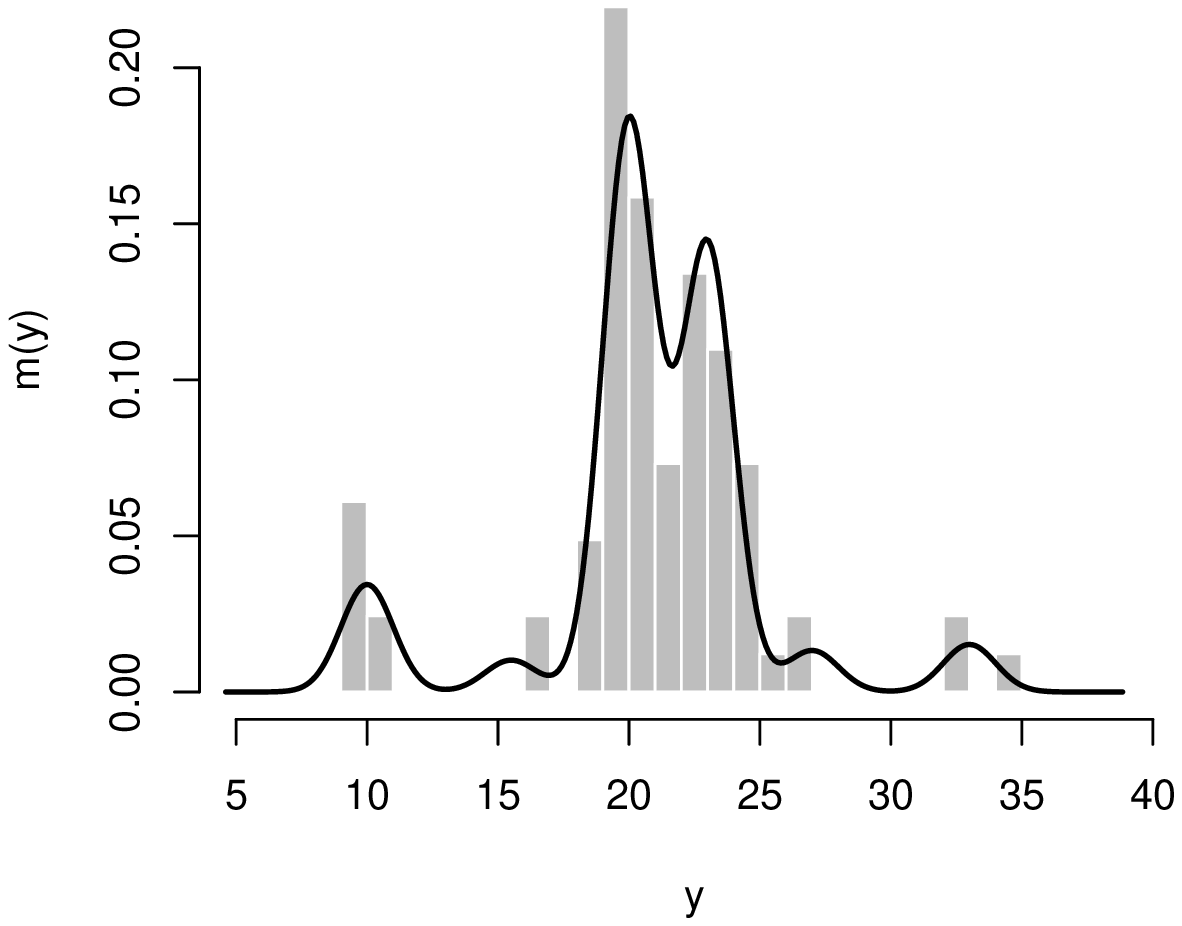}}}
\caption{Plots of the PR estimates for the galactic velocity data in Example~\ref{ex:galaxy1}.}
\label{fig:galaxy}
\end{center}
\end{figure}

\begin{ex}
\label{ex:poismix}
\citet[][Table~1]{karlis.xekalaki.2001} present data on the number of defaulted installments in a Spanish financial institution.  This data has a high number of zero counts, as well as substantial overdispersion.  This suggests a Poisson mixture model, and here we compare the PR-based estimates to others presented in the literature.  The first three rows of Table~\ref{table:poismix.est} show the estimates of $(f,U)$ for three methods in an zero-inflated Poisson mixture model.  These include an estimate based on the AIC penalty, the SCAD-based penalized likelihood approach of \citet{chen.khalili.2008}, and a minimum Hellinger distance method for count data \citep{woosriram2007}.  I start by bounding the support by $\Ubar = [0,30]$ and taking $\U$ to be a set of 100 equispaced points in $\Ubar$.  All but the Woo--Sriram estimates have five support points, including zero.  Besides this, we find that the corresponding estimates are quite similar.  An attractive feature of this method is that no special adjustments are needed for zero-inflation.  That is, zero-inflation can be achieved by simply including zero in the grid $\U$ and letting the data decide if a mass at zero is appropriate.  Fitted values were obtained for each of the four methods (not shown) and I find that, for small $y$-values, where the observed counts are relatively large, the PR-based estimate appears to provide a better overall fit compared to the others.  
\end{ex}

\begin{table}
\begin{center}
\begin{tabular}{c|cccccc}
\hline
Estimates & $(u_1, f(u_1))$ & $(u_2, f(u_2))$ & $(u_3, f(u_3))$ & $(u_4, f(u_4))$ & $(u_5, f(u_5))$ \\
\hline
AIC--BIC & (0, .314) & (.298, .435) & (4.37, .200) & (10.99, .048) & (26.51, .002) \\
MSCAD & (0, .328) & (.302, .417) & (4.19, .193) & (9.78, .055) & (20.01, .007) \\
WS & (0, .373) & (.36, .385) & (4.52, .199) & (11.26, .043) & \\
SASA & (0, .328) & (.303, .418) & (4.24, .201) & (10.91, .051) & (27.27, .002) \\
\hline
\end{tabular}
\end{center}
\caption{Estimates of $(f,U)$ for the financial data Poisson mixture in Example~\ref{ex:poismix}.  The first three rows are taken from \citet[][Table~10]{chen.khalili.2008}.}
\label{table:poismix.est}
\end{table}

\ifthenelse{1=1}{}{
\begin{table}
\begin{center}
\begin{tabular}{ccccc}
\hline
Defaults & Obs.~Freq. & MSCAD Exp.~Freq. & WS Exp.~Freq. & SASA Exp.~Freq. \\
\hline
0 & 3,002 & 2,998.6 & 3,019.9 & 3003.6 \\
1 & 502 & 494.4 & 499.6 & 496.6 \\
2 & 187 & 187.0 & 185.7 & 188.5 \\
3 & 138 & 177.0 & 166.9 & 179.6 \\
4 & 233 & 182.2 & 179.4 & 185.5 \\
5 & 160 & 158.5 & 163.8 & 160.5 \\
6 & 107 & 120.8 & 127.8 & 119.7 \\
7 & 80 & 86.5 & 89.6 & 82.3 \\
8 & 59 & 62.6 & 60.6 & 56.9 \\
9 & 53 & 48.1 & 42.9 & 42.9 \\
10 & 41 & 38.7 & 33.4 & 35.7 \\
11 & 28 & 31.4 & 28.1 & 31.1 \\
12 & 34 & 24.7 & 24.1 & 26.8 \\
13 & 10 & 18.7 & 20.0 & 22.0 \\
14 & 13 & 13.6 & 15.9 & 17.0 \\
15 & 11 & 9.7 & 11.8 & 12.3 \\
$\geq 16$ & 33 & 38.3 & 21.5 & 28.9 \\
\hline
\end{tabular}
\end{center}
\caption{Fitted values for the financial data Poisson mixture in Example~\ref{ex:poismix}.  The MSCAD and WS results are taken from \citet[][Table~11]{chen.khalili.2008}. }
\label{table:poismix.fit}
\end{table}
}


\section*{Acknowledgments}

The author thanks Professor Surya Tokdar for a number of helpful suggestions, and the Department of Mathematical Sciences, Indiana University--Purdue University Indianapolis, for their hospitality when a portion of this work was completed.

\appendix

\section{Convergence rates for stochastic approximation}
\label{S:chen}

Consider a stochastic approximation process $\{X_n: n \geq 0\}$ which, for fixed initial value $X_0 = x_0$, is defined recursively as follows:
\[ X_n = X_{n-1} + a_n \phi(X_{n-1}) + a_n Z_n, \quad n \geq 1. \] 
The process is designed so that $X_n \to x^\star$ almost surely, where $x^\star$ satisfies $\phi(x^\star) = 0$.  We shall assume that $\{X_n\}$ bounded; otherwise, some truncation or projection techniques are needed \citep{chen2002, kushner}.  The PR estimates $f_n$ are constrained to the simplex, so they satisfy this boundedness condition trivially.  Next are the main assumptions of the theorem.  

\begin{itemize}

\item[A1.] The weights $\{a_n\}$ satisfy $a_n > 0$, $a_n \to 0$, $\sum_n a_n = \infty$, and $a_{n+1}^{-1} - a_n^{-1} \to \alpha$ for some $\alpha \geq 0$.    
\vspace{-2mm}
\item[A2.] There exists a Lyapunov function $\ell(x)$ at the equilibrium point $x^\star$ of the ODE $dx_t/dt = \phi(x_t)$.  
\vspace{-2mm}
\item[A3.] $\sum_n a_n^{1-\delta} Z_n < \infty$ almost surely for some $\delta \in (0,1/2)$.  
\vspace{-2mm}
\item[A4.] $\phi(x)$ is continuously differentiable, and all eigenvalues of $J+\alpha\delta I$ have negative real parts, where $J = D\phi(x^\star)$.

\end{itemize}

\begin{chenthm}
Under A1--A4, $\|X_n-x^\star\| = o(a_n^\delta)$ almost surely.  
\end{chenthm}

\bibliographystyle{/Users/rgmartin/Research/TexStuff/asa}
\bibliography{/Users/rgmartin/Research/mybib}

\end{document}